\documentclass[12pt]{amsart}
\usepackage{amsfonts}
\usepackage{amssymb,amsmath,amscd}
\usepackage{graphicx}
\usepackage{mathrsfs}

\newtheorem{thm}{\bf Theorem}[section]
\newtheorem{prop}[thm]{\sc Proposition}
\newtheorem{lem}[thm]{\sc Lemma}

\theoremstyle{definition}

\theoremstyle{definition}

\theoremstyle{definition}
\newtheorem{rem}[thm]{\sc Remark}
\theoremstyle{definition}

\theoremstyle{definition}

\theoremstyle{definition}

\numberwithin{equation}{section}

\begin{document}
\title[Connected sum]{Connected sum of CR manifolds with positive CR Yamabe
constant}
\author{Jih-Hsin Cheng}
\address{Institute of Mathematics, Academia Sinica and National Center for
Theoretical Sciences, Taipei, Taiwan, ROC}
\email{cheng@math.sinica.edu.tw }
\author{Hung-Lin Chiu}
\address{Department of Mathematics, National Tsing-Hua University, Hsinchu,
Taiwan, ROC}
\email{hlchiu@math.nthu.edu.tw}
\author{Pak Tung Ho}
\address{Department of Mathematics, Sogang University, Seoul, Korea}
\email{paktung@yahoo.com.hk, ptho@sogang.ac.kr}

\begin{abstract}
Suppose $M_{1}$ and $M_{2}$ are $3$-dimensional closed (compact without
boundary) CR manifolds with positive CR Yamabe constant. In this note, we
show that the connected sum of $M_{1}$ and $M_{2}$ also admits a CR
structure with positive CR Yamabe constant.
\end{abstract}

\maketitle

%\subjclass{1991 Mathematics Subject Classification. Primary: 35L80;Secondary: 35J70, 32V20, 53A10, 49Q10.}
%\keywords{Key Words: Heisenberg group, umbilicity, Pansu sphere}

\section{Introduction}

In Riemannian geometry, the scalar curvature is the simplest curvature
invariant of a Riemannian manifold. It was shown by Gromov-Lawson in \cite%
{Gromov&Lawson} and independently by Schoen-Yau in \cite{Schoen&Yau} that
the connected sum of two closed (that is, compact without boundary)
manifolds of positive scalar curvature has a metric of positive scalar
curvature. It was also shown by Schoen-Yau in \cite{Schoen&Yau} that the
connected sum of two closed conformally flat manifolds of positive scalar
curvature has a conformally flat metric of positive scalar curvature (see
Corollary 5 in \cite{Schoen&Yau}). In view of the similarity between the
scalar curvature in Riemannian geometry and the Tanaka-Webster scalar
curvature in CR geometry, it would be natural to ask if the corresponding
results hold for the Tanaka-Webster scalar curvature. It is the purpose of
this note to answer this question.

For basic materials in CR geometry and pseudohermitian geometry, we refer
the readers to \cite{DT}, \cite{Jerison&Lee}, \cite{Lee} or \cite{Webster}.
and the references therein. Let $(M,J)$ be a closed, strictly pseudoconvex
CR manifold of dimension $2n+1$. 
Take a contact form $%
\theta$, which means a $1$-form satisfying the complete non-integrability
condition: $\theta\wedge(d\theta)^{n}\neq 0$ for each point of $M$. Let $%
\xi=\ker{\theta}$, which is the associated contact bundle. We choose $\theta$
compatible with $J$ in the following sense: the CR structure $J$ is defined
on $\xi$ and $d\theta(X,JX)>0$ for any nonzero vector $X\in\xi$. The Levi
metric (or form) of $\theta$ is the real symmetric bilinear form $L_{\theta}$
on $\xi$ defined by 
\begin{equation*}
L_{\theta}(X,Y)=2d\theta(X,JY),\ \ \ X,Y\in\xi.
\end{equation*}
$L_{\theta}$ extends by complex linearity to $\xi\otimes C$, and induces a
hermitian metric (or form) on the subbundle $\xi_{1,0}\subset\xi\otimes C$
of all CR holomorphic vectors (Note that, instead of $\xi_{1,0}$, D. Jerison
and J. Lee in \cite{Jerison&Lee} used $T_{1,0}$ to denote the CR holomorphic
subbundle). For a real function $u$, the subgradient of $u$ is denoted by $%
\nabla_{b}u$ and defined as the unique vector in $\xi$ satisfying 
\begin{equation*}
Xu=L_{\theta}(X,\nabla_{b}u),
\end{equation*}
for all $X\in\xi$. Here $Xu$ means the directional derivative of $u$ along $%
X $. The norm of $\nabla_{b}u$ is defined by 
\begin{equation*}
|\nabla_{b}u|^{2}_{\theta}=L_{\theta}(\nabla_{b}u,\nabla_{b}u).
\end{equation*}
In \cite{DT}, S. Dragomir and G. Tomassini considered the gradient $\nabla u$
of $u$ with respect to the Webster metric $g_{\theta}$. It is easy to see
that $\nabla_{b}u=\pi_{H}\nabla u$, where $\pi_{H}:TM\rightarrow \xi$ is the
natural orthogonal projection defined in \cite{DT} (in which, instead of $%
\xi $, they used $H(M)$ to denote the contact bundle). It is also easy to
check that 
\begin{equation*}
L_{\theta}(\nabla_{b}u,\nabla_{b}v)=L_{\theta}^{*}(du,dv)=L_{%
\theta}^{*}(d_{b}u,d_{b}v),
\end{equation*}
for any real functions $u$ and $v$, where $L_{\theta}^{*}$ is the induced
metric on $\xi^{*}$, determined by $L_{\theta}$, and extends naturally to $%
T^{*}M$ (see \cite{Jerison&Lee} for more details) and $d_{b}=\bar{\partial}%
_{b}+\partial_{b}$. 
Take the volume form $dV:=\theta \wedge (d\theta )^{n}.$ The sublaplacian
operator $\Delta_{b}$ is defined on real functions $u\in C^{\infty}(M)$ by 
\begin{equation*}
\int_{M}(\Delta_{b}u)vdV=\int_{M}L_{\theta}^{*}(du,dv)dV,
\end{equation*}
for all $v\in C^{\infty}_{0}(M)$.

If $\theta$ is replaced by $\tilde{\theta}=u^{p-2}\theta$, with $p=2+\frac{2%
}{n}$, then we have the transformation law for Tanaka-Webster scalar
curvatures 
\begin{equation*}
\tilde{R}=u^{1-p}(b_{n}\Delta_{b}+R)u
\end{equation*}
where $b_{n}=2+\frac{2}{n}$, and $R$ (or $R_{J,\theta}$) and $\tilde{R}$ (or 
$R_{J,\tilde\theta}$) are respectively Tanaka-Webster scalar curvature on
the pseudohermitian manifold $(M,J,\theta )$ and $(M,J,\tilde{\theta} )$.

 We define the CR
Yamabe constant $\lambda (M,J)$ (or $\lambda (M)$ if $J$ is clear in the
context) as follows: (see \cite{Jerison&Lee})
\begin{equation*}
\lambda (M,J)=\inf_{u>0}\frac{E_{\theta }(u)}{(\int_{M}u^{2+\frac{2}{n}%
}dV_{\theta })^{\frac{n}{n+1}}},
\end{equation*}%
where
\begin{equation*}
E_{\theta }(u)=\int_{M}\left( (2+\frac{2}{n})|\nabla
_{b}u|^{2}+Ru^{2}\right) dV_{\theta }.
\end{equation*}%
Similar to the Riemannian case, one can show that $\lambda (M,J)>0$ if and
only if there exists a contact form $\tilde{\theta}$ conformal to $\theta $
such that the Tanaka-Webster scalar curvature of $\tilde{\theta}$ is
positive.

In \cite{Cheng&Chiu}, the first and the second authors proved the following
theorem, which is the CR version of Schoen-Yau's result mentioned above (see
also \cite{Kobayashi} for a different proof by O. Kobayashi). Recall that a CR manifold is called  spherical if it is locally CR
isomorphic to the standard CR sphere $S^{2n+1}$.

\begin{thm}
$($\cite{Cheng&Chiu}$)$\label{thm1} Suppose $(M_1, J_1)$ and $(M_2, J_2)$
are two closed, spherical CR manifolds of dimension $2n + 1$ with $%
\lambda(M_k, J_k) > 0$ for $k = 1, 2$. Then their connected sum $M_1\# M_2$
admits a spherical CR structure $J$ with $\lambda(M_1\# M_2, J) > 0$.
\end{thm}

The idea of the proof of Theorem \ref{thm1} was motivated by the work of O.
Kobayashi in \cite{Kobayashi}. More precisely, we fix a point $p_{j}\in
M_{j} $ for $j=1,2$. We first take off two small balls around $p_{1}$ and $%
p_{2}$. Since $M_{j}$ are spherical, we can attach the Heisenberg cylinder
in each of punched neighborhood of $p_{j}$. We then glue two Heisenberg
cylinders together to get a spherical CR manifold.

In this note, we continue our study on the Tanaka-Webster scalar curvature
of connected sum on CR manifolds without assuming they are spherical.
In particular, we prove the following
theorem, which can be viewed as the analogous result of Gromov-Lawson and of
Schoen-Yau mentioned above. \newline

\textbf{Theorem A.} Suppose $(M_1, J_1)$ and $(M_2, J_2)$ are two $3$%
-dimensional closed CR manifolds with $\lambda(M_k, J_k) > 0$ for $k = 1, 2$%
. Then their connected sum $M_1\# M_2$ admits a CR structure $J$ with $%
\lambda(M_1\# M_2, J) > 0$. \newline

Note that the above argument for the spherical case cannot be applied
directly, since we cannot attach the Heisenberg cylinder to the punched
neighborhood of a point. However, in this paper, we will mainly construct a
new CR structure, through the deformation tensor, which outside a ball is the
given CR structure and is
spherical in a neighborhood contained in the ball. In addition, we can
construct such a CR structure such that its Yamabe constant is as close as
possible to the one of the given CR structure. Hence, together with Theorem %
\ref{thm1}, we obtain Theorem A. The idea of the proof is elegant and also, as Theorem \ref{thm1}, 
motivated by the work of Kobayashi in \cite{Kobayashi}. However, due to different nature of geometric structures,
the way to construct the new CR structure is entirely different from that in Riemannian geometry. We use the concept 
of the deformation tensor, which has been well studied in $3$-dimensional CR geometry.

We learned that Dietrich \cite{Dietrich}, among others, also claimed the same
statement as in Theorem A. But the proof of his key lemma (Lemma 5.6 in
\cite{Dietrich}) is not clear to us (cf. the proof of Proposition
3.6 in this paper).

\textbf{Acknowledgments. }J.-H. Cheng (resp. H.-L. Chiu) would like to
thank the Ministry of Science and Technology of Taiwan, R.O.C. for the
support of the project: MOST 107-2115-M-001-011- (resp. MOST
106-2115-M-007-017-MY3). J.-H. Cheng would also like to thank the
National Center for Theoretical Sciences for the constant support.

\section{Basic Material}

For basic material in $CR$ and pseudohermitian geometry, we refer the reader
to \cite{DT}, \cite{Jerison&Lee}, \cite{Lee} or \cite{Webster}. Let $%
(M^{3},J,\theta )$ be a pseudohermitian manifold. In \cite{Webster}, S.
Webster showed that there is a natural connection in the bundle $\xi _{1,0}$
of all CR holomorphic vectors adapted to the pseudohermitian structure $%
(J,\theta )$. To define the connection, choose an \textbf{orthonormal}
admissible coframe $\{\theta ^{1}\}$ and dual frame $\{Z_{1}\}$ for $\xi
_{1,0}$. Webster showed that there are uniquely determined $1$-forms $\theta
_{1}{}^{1},\tau ^{1}$ on $M$ satisfying the following structure equations
\begin{equation}
\begin{split}
d\theta ^{1}& =\theta ^{1}\wedge \theta _{1}{}^{1}+\theta \wedge \tau ^{1},
\\
0& =\theta _{1}{}^{1}+\theta _{\bar{1}}{}^{\bar{1}}, \\
0& =\tau _{1}\wedge \theta ^{1},
\end{split}
\label{steq}
\end{equation}%
in which $\tau _{1}=\tau ^{\bar{1}}$. The forms $\theta _{1}{}^{1},\tau ^{1}$
are called the pseudohermitian connection form and torsion form,
respectively. Recall that the Heisenberg group $H_{1}$ is the space $\mathbb{%
R}^{3}$ endowed with the group multiplication
\begin{equation*}
(x_{1},y_{1},z_{1})\circ
(x_{2},y_{2},z_{2})=(x_{1}+x_{2},y_{1}+y_{2},z_{1}+z_{2}+y_{1}x_{2}-x_{1}y_{2}),
\end{equation*}%
which is a $3$-dimensional Lie group. The space of all left invariant vector
fields is spanned by the following three vector fields:
\begin{equation*}
\mathring{e}_{1}=\frac{\partial }{\partial x}+y\frac{\partial }{\partial z}%
,~~\mathring{e}_{2}=\frac{\partial }{\partial y}-x\frac{\partial }{\partial z%
}~~\mbox{ and }~~T=\frac{\partial }{\partial z}.
\end{equation*}%
The standard contact bundle on $H_{1}$ is the subbundle $\mathring{\xi}$ of
the tangent bundle $TH_{1}$, which is spanned by $\mathring{e}_{1}$ and $%
\mathring{e}_{2}$. It can also be equivalently defined as the kernel of the
contact form
\begin{equation*}
\Theta =dz+xdy-ydx.
\end{equation*}%
The CR structure on $H_{1}$ is the endomorphism $\mathring{J}:\mathring{\xi}%
\rightarrow \mathring{\xi}$ defined by
\begin{equation*}
\mathring{J}(\mathring{e}_{1})=\mathring{e}_{2}~~\mbox{ and }~~J(\mathring{e}%
_{2})=-\mathring{e}_{1}.
\end{equation*}%
One can view $H_{1}$ as a pseudohermitian manifold with the standard
pseudohermitian structure $(\mathring{J},\Theta)$. In the Heisenberg
group $H_{1}$, relative to the standard left invariant frame $\mathring{Z}%
_{1}=\frac{1}{2}(\mathring{e}_{1}-i\mathring{e}_{2})$ (dual coframe is $%
\mathring{\theta}^{1}=dx+idy$), it is easy to see that both forms $\theta
_{1}{}^{1}$ and $\tau ^{1}$ \textbf{vanish}.

\subsection{The deformation tensor}

Suppose $J$ is a CR structure compatible with $\Theta $ in the following sense: it is
defined on $\mathring{\xi}$ such that
\begin{equation}\label{2.2}
d\Theta (X,JX)>0\mbox{ for any nonzero vector }X\in \mathring{\xi}.
\end{equation}
Let $Z_{1}$ be a CR holomorphic vector field relative
to $J$. We express it as
\begin{equation}\label{2.3}
Z_{1}=a_{1}{}^{1}\mathring{Z}_{1}+b_{1}{}^{\bar{1}}\mathring{Z}_{\bar{1}},
\end{equation}%
for some function $a_{1}{}^{1},b_{1}{}^{\bar{1}}$. We compute
\begin{equation}\label{orex}
Z_{1}\wedge Z_{\bar{1}}=(|a_{1}{}^{1}|^{2}-|b_{1}{}^{\bar{1}}|^{2})\mathring{%
Z}_{1}\wedge \mathring{Z}_{\bar{1}}.
\end{equation}%
The compatibility of $J$ with $\Theta$ in (\ref{2.2})
implies that $d\Theta(Z_{1}, JZ_{1})>0$,
which together with (\ref{orex}) implies
that
\begin{equation}\label{2.4}
|a_{1}{}^{1}|^{2}>|b_{1}{}^{\bar{1}}|^{2}.
\end{equation}
In particular, we have $%
a_{1}{}^{1}\neq 0$. Thus we can define
\begin{equation}\label{2.5}
\phi =(a_{\bar{1}}{}^{\bar{1}})^{-1}b_{\bar{1}}{}^{1},
\end{equation}%
where $a_{\bar{1}}{}^{\bar{1}},b_{\bar{1}}{}^{1}$ is the conjugate of $%
a_{1}{}^{1},b_{1}{}^{\bar{1}}$, respectively. We call $\phi $ \textbf{the
deformation tensor} of $J$ (note that $\phi $ depends on frames. It behaves
as a tensor when changing frames. For notational simplicity, we suppress its
tensor indices). Thus (\ref{2.4})
 implies that $|\phi |<1$.
 It follows from (\ref{2.3}) and (\ref{2.5})
 that any CR anti-holomorphic vector field $Z_{\bar{1}}$ has the form $%
Z_{\bar{1}}=a_{\bar{1}}{}^{\bar{1}}(\mathring{Z}_{\bar{1}}+\phi \mathring{Z}%
_{1})$, for some function $a_{\bar{1}}{}^{\bar{1}}$. Conversely, any
function $\phi $ with $|\phi |<1$ defines a CR structure $J$ compatible with
$\Theta $ by regarding $Z_{\bar{1}}=a_{\bar{1}}{}^{\bar{1}}(\mathring{Z}_{%
\bar{1}}+\phi \mathring{Z}_{1})$ as its corresponding CR anti-holomorphic
vector field.

\section{Proof}

Let $(M^3, J,\theta)$ be a pseudohermitian manifold. To prove \textbf{%
Theorem A}, first we would like to construct a sequence of pseudohermitian
structures $\{(J_{i},\theta_{i})\}$ such that $\{(J_{i},\theta_{i})\}$
converges to $(J,\theta)$ in $C^0$ and the corresponding Tanaka-Webster scalar curvature $%
R_{i}$ also converges to the Tanaka-Webster scalar curvature $R$ of $(J,\theta)$ in $C^0$. This is Proposition \ref{mainpro}. In addition, each CR structure $J_{i}$ we construct in Proposition \ref{mainpro} is CR spherical around a given point $p\in M$. Then, together with Proposition \ref{proflim} and Theorem \ref{thm1}, we obtain {\bf Theorem A}.

For Proposition \ref{mainpro}, we construct such a sequence as follows:
For each $p\in M$, there exists a neighborhood $U$ of $p\in M$ which is
contactomorphic to a neighborhood $V$ of $0\in H_{1}$. Let $%
\Phi:U\rightarrow V$ be such a contactomorphism and $\Phi(p)=0$, we identify
$U$ with $V$ under $\Phi$. Then, on $U$ (or on $V$), the CR structure $J$
can be represented by a deformation tensor $\phi$ with $|\phi|<1$ such that $%
\mathring{Z}_{\bar 1}+\phi\mathring{Z}_{1}$ is a CR anti-holomorphic vector
field. In addition, it is easy to see that one can take a contactomorphism $%
\Phi$ with $\Phi(p)=0$ such that the deformation function $\phi$ satisfies $%
\phi(0)=0\ \text{and}\ \phi_{1}(0)=\phi_{\bar 1}(0)=0$, where $\phi_{1}=%
\mathring{Z}_{1}\phi$ and $\phi_{\bar 1}=\mathring{Z}_{\bar1}\phi$. Thus, we
can assume, without loss of generality, that
\begin{equation}\label{3.1}
\begin{split}
&\theta|_{U}=\Theta,\\
&\phi(0)=\phi_{1}(0)=\phi_{\bar 1}(0)=0.
\end{split}
\end{equation}
Relative to the contact form $\Theta$,
\begin{equation}  \label{defof}
Z_{\bar 1}=\left(\frac{1}{1-|\phi|^{2}}\right)^{1/2}(\mathring{Z}_{\bar
1}+\phi\mathring{Z}_{1}),
\end{equation}
is a unit vector field. By (\ref{defof}), the dual coframe is
\begin{equation}
\theta^{1}=\left(\frac{1}{1-|\phi|^{2}}\right)^{1/2}(\mathring{\theta}%
^{1}-\phi\mathring{\theta}^{\bar 1}).
\end{equation}
On $U$, we can express the pseudohermitian connection form, torsion form, Tanaka-Webster curvature and sub-laplacian of $(J,\Theta)$ in terms of objects of the Heisenberg group as Propositions \ref{foofin1} and \ref{foofin2} specify.
These expressions help us construct a pseudohermitian sequence we want.

\begin{prop}\label{foofin1}
Let $\theta_{1}{}^{1}$ and $\tau^{1}=A^{1}{}_{\bar 1}\theta^{\bar 1}$ be the
pseudohermitian connection form and torsion form relative to $\theta^{1}$,
respectively. Then we have
\begin{equation}  \label{trlaw00}
\begin{split}
A^{1}{}_{\bar 1}&=-\frac{\phi_{0}}{1-|\phi|^{2}}; \\
\theta_{1}{}^{1}&=-d\ln{\left(\frac{1}{1-|\phi|^{2}}\right)}^{1/2}+\left[%
\frac{\bar{\phi}\phi_{0}}{1-|\phi|^{2}}\right]\Theta \\
&\hspace{4mm}+\left[\frac{\bar{\phi}_{\bar 1}+\bar{\phi}\phi_{1}}{1-|\phi|^{2}}+\bar{\phi%
}\mathring{Z}_{\bar 1}\left(\frac{1}{1-|\phi|^{2}}\right)+\mathring{Z}%
_{1}\left(\frac{1}{1-|\phi|^{2}}\right)\right]\mathring{\theta}^{1} \\
&\hspace{4mm}-\left[\frac{\phi_{1}+\phi\bar{\phi}_{\bar 1}}{1-|\phi|^{2}}+|\phi|^{2}%
\mathring{Z}_{\bar 1}\left(\frac{1}{1-|\phi|^{2}}\right)+\phi\mathring{Z}%
_{1}\left(\frac{1}{1-|\phi|^{2}}\right)\right]\mathring{\theta}^{\bar 1},
\end{split}%
\end{equation}
where all the derivatives are computed in $H_{1}$; for example, $\phi_{1}=%
\mathring{Z}_{1}\phi,\ \phi_{0}=T\phi$, and so on.
\end{prop}

\begin{proof}
One can check directly that $\theta_{1}{}^{1}$ and $A^{1}{}_{\bar 1}$ in (%
\ref{trlaw00}) satisfy the structure equations (\ref{steq}). And by
uniqueness, we complete the proof.
\end{proof}
\begin{prop}\label{foofin2}
Let $R^{\phi,\Theta}$ and $\Delta_{b}^{\phi,\Theta}$ be the Webster
curvature and (negative) sub-laplacian of $(J,\Theta)$ on $U$, respectively.
Then we have
\begin{equation}  \label{trlaw01}
\begin{split}
R^{\phi,\Theta}&=-\mathring{Z}_{\bar 1}\left[\frac{\bar{\phi}_{\bar 1}+\bar{%
\phi}\phi_{1}}{1-|\phi|^{2}}+\bar{\phi}\mathring{Z}_{\bar 1}\left(\frac{1}{%
1-|\phi|^{2}}\right)+\mathring{Z}_{1}\left(\frac{1}{1-|\phi|^{2}}\right)%
\right] \\
&\hspace{4mm}-\mathring{Z}_{1}\left[\frac{\phi_{1}+\phi\bar{\phi}_{\bar 1}}{1-|\phi|^{2}}%
+|\phi|^{2}\mathring{Z}_{\bar 1}\left(\frac{1}{1-|\phi|^{2}}\right)+\phi%
\mathring{Z}_{1}\left(\frac{1}{1-|\phi|^{2}}\right)\right] \\
&\hspace{4mm}+\frac{i\bar{\phi}\phi_{0}}{1-|\phi|^{2}},
\end{split}%
\end{equation}
and
\begin{equation}  \label{trlaw02}
\begin{split}
\Delta_{b}^{\phi,\Theta}u&=\left[\frac{1+|\phi|^2}{1-|\phi|^2}\right]%
\mathring{\Delta}_{b}u-\left[\frac{2\bar{\phi}}{1-|\phi|^2}\right]u_{\bar{1}%
\bar{1}}-\left[\frac{2\phi}{1-|\phi|^2}\right]u_{11} \\
&\hspace{4mm}-\left[\frac{2\bar{\phi}_{\bar 1}+|\phi|^{2}_{1}}{1-|\phi|^{2}}+2\bar{\phi}%
\mathring{Z}_{\bar 1}\left(\frac{1}{1-|\phi|^{2}}\right)+(1+|\phi|^{2})%
\mathring{Z}_{1}\left(\frac{1}{1-|\phi|^{2}}\right)\right]u_{\bar 1} \\
&\hspace{4mm}-\left[\frac{2\phi_{1}+|\phi|^{2}_{\bar 1}}{1-|\phi|^{2}}+(1+|\phi|^{2})%
\mathring{Z}_{\bar 1}\left(\frac{1}{1-|\phi|^{2}}\right)+2\phi\mathring{Z}%
_{1}\left(\frac{1}{1-|\phi|^{2}}\right)\right]u_{1}.
\end{split}%
\end{equation}
\end{prop}

\begin{proof}
Recall that S. Webster \cite{Webster} showed that $d\theta _{1}{}^{1}$ can be written as
\begin{equation}
d\theta _{1}{}^{1}=R\theta ^{1}\wedge \theta ^{\bar{1}}\ \ \text{mod}\
\theta ,
\end{equation}%
where $R$ is the Tanaka-Webster scalar curvature. Since $\theta ^{1}$ is an unit coframe,
we have $\theta ^{1}\wedge \theta ^{\bar{1}}=-id\Theta =\mathring{\theta}%
^{1}\wedge \mathring{\theta}^{\bar{1}}$. On the other hand, $d\mathring{%
\theta}^{1}=0$. Hence, (\ref{trlaw01}) follows immediately from (\ref{trlaw00}). For (\ref{trlaw02}), recall that
\begin{equation}
\Delta _{b}^{\phi ,\Theta }u=-\Big(Z_{\bar{1}}Z_{1}u-\theta _{1}{}^{1}(Z_{%
\bar{1}})Z_{1}u\Big)+\ \text{conjugate},
\end{equation}%
and (\ref{trlaw02}) is just a straightforward computation in terms of (\ref%
{trlaw00}).
\end{proof}

\begin{rem}
The first author and I. H. Tsai deduced a more general formula for the
Tanaka-Webster scalar curvature (see (4.6) in \cite{CT}).
\end{rem}

To construct the sequence we want, we also need the following lemma which is a standard
result in the literature (see \cite{Dietrich,Kobayashi}).

\begin{lem}\label{lem1}
\label{kl} For any $\delta>0$, there is a nonnegative function $%
\chi_{\delta}\in C^{\infty}(\mathbb{R})$ such that \newline
(i) $0\leq\chi_{\delta}\leq 1,\ \chi_{\delta}(t)\equiv 1$ in a neighborhood
of $0$ and $\chi_{\delta}(t)\equiv 0$ for $|t|\geq\delta$; \newline
(ii) $|\chi_{\delta}'(t)|\leq\delta t^{-1}$ and $%
|\chi_{\delta}''(t)|\leq\delta t^{-2}$ for all $t$.
\end{lem}

Now, for each $\delta>0$, by means of the cut-off function $\chi_{\delta}$, we define the CR structure $J^{\delta}$ with the corresponding deformation tensor $\phi ^{\delta }=(1-\chi _{\delta }(\rho ))\phi$. It is easy to see that $J^{\delta}=J$ outside the $\delta$-ball $B(\delta)$ centered at $0$ and $J^{\delta}=\mathring{J}$ in a neighborhood of $0$, which is CR spherical. If we, in addition, consider $\theta^{\delta}=\Theta$, then it is easy to see that the sequence $(J^{\delta},\theta^{\delta})$ converges to $(J,\theta)$ in $C^{0}$ (Note that we have chosen $\theta$ such that $\theta|_{U}=\Theta$). However, in general, since $R^{\phi,\Theta}(0)$ may not be zero, the corresponding Tanaka-Webster curvature of $(J^{\delta},\theta^{\delta})$  does not converge to the one of $(J,\theta)$. In order to have a sequence we want in Proposition \ref{mainpro}, we need to deform the contact form $\Theta$.
 
 Recall that if we consider the new contact form $\theta^{u}=u^{2}\Theta$, then, on $U$,
we have the following transformation law of the Tanaka-Webster scalar  curvature: (for the details,
see \cite{DT, Jerison&Lee})
\begin{equation}  \label{trlaw03}
R^{\phi,\theta^{u}}=u^{-3}(4\Delta^{\phi,\Theta}_{b}u+R^{\phi,\Theta}u),
\end{equation}
where $R^{\phi,\theta^{u}}$ is the Tanaka-Webster scalar curvature with respect to $(J,\theta^{u})=(J,u^{2}\Theta)$.

On the other hand, the standard CR structure of the Heisenberg group on $U$
is represented by the zero deformation function $\phi\equiv 0$, which is CR
\textbf{spherical}. Let $u$ be a positive function in a neighborhood of $0$
such that $u(0)=1, (\mathring{Z}_{1}u)(0)=(\mathring{Z}_{\bar 1}u)(0)=0$ and
\begin{equation}  \label{trlaw040}
R^{0,\theta^{u}}(0)=R^{\phi,\Theta}(0).
\end{equation}
It follows from (\ref{trlaw03}) and (\ref{trlaw040}) that
\begin{equation}  \label{trlaw041}
R^{0,\theta^{u}}=u^{-3}(4\mathring{\Delta}_{b}u).
\end{equation}

\subsection{Construction of a sequence $(J^{\protect\delta},\protect\theta^{%
\protect\delta})$}

We are now ready to construct  a sequence of pseudohermitian structures we describe in the
beginning of this section. First, we re-formulate (\ref{trlaw01}) and (\ref{trlaw02}%
) as what we need.

\begin{prop}
Let $F=F(\phi)=\left(\frac{1}{1-|\phi|^{2}}\right)^{\frac{1}{2}}$. We have
\begin{equation}  \label{pbf1}
\begin{split}
R^{\phi,\Theta}&=-F^{2}(\phi_{11}+\bar{\phi}_{\bar{1}\bar
1})+\sum_{a,b\in\{1,\bar{1}\}}P_{ab}\phi_{ab}+\sum_{a,b\in\{1,\bar{1}%
\}}Q_{ab}\bar{\phi}_{ab}+P, \\
\Delta_{b}^{\phi,\Theta}u&=F^{2}\mathring{\Delta}_{b}u+\sum_{a,b\in\{1,\bar{1%
}\}}S_{ab}u_{ab}+Su_{1}+\bar{S}u_{\bar 1},
\end{split}%
\end{equation}
where $P_{ab},Q_{ab},P,S_{ab}$ and $S$ are all polynomials in $F,\phi,\bar{%
\phi},\phi_{1},\phi_{\bar 1},\bar{\phi}_{1},\bar{\phi}_{\bar 1}$ such that
\begin{equation}  \label{pbf2}
P_{ab}(0)=Q_{ab}(0) =P(0)=S_{ab}(0)=S(0)=0.
\end{equation}
\end{prop}

Since $F(0)=1$, condition (\ref{pbf2}) means that each polynomial does not
include monomial terms $F^{k}$ for some nonnegative integer $k$.

Now, for each $\delta>0$, we define a pseudohermitian structure $(J^{\delta
},\theta ^{\delta })$ by
\begin{equation}\label{3.14}
\begin{split}
\phi ^{\delta }& =(1-\chi _{\delta }(\rho ))\phi ,\ \ \text{which is the
deformation tensor of}\ J^{\delta }; \\
\theta ^{\delta }& =(1-\chi _{\delta }(\rho ))\Theta +\chi _{\delta }(\rho
)(\theta ^{u})=(v^{\delta })^{2}\Theta, \text{ where }\theta ^{u}=u^{2}\Theta~~\mbox{ and }~~(v^{\delta })^{2}=1+\chi _{\delta
}(u^{2}-1).
\end{split}%
\end{equation}%
Here, $\chi _{\delta}$ is the function constructed in Lemma \ref{kl}.
It follows from Lemma \ref{kl} that $(\phi ^{\delta },\theta ^{\delta })=(\phi ,\Theta )$
outside the $\delta$-ball $B(\delta )$ centered at $0$, and $(\phi ^{\delta
},\theta ^{\delta })=(0,\theta ^{u})$ in a neighborhood of $0$. Moreover,
by (\ref{3.1}), we have
\begin{equation}
\begin{split}
& R^{0,\theta ^{u}}(0)=R^{\phi ,\Theta }(0), \\
& \phi (0)=\phi _{1}(0)=\phi _{\bar{1}}(0)=0, \\
& u(0)=1,\ u_{1}(0)=u_{\bar{1}}(0)=0.
\end{split}
\label{basass}
\end{equation}%
Notice that we have used the cut-off function $\chi _{\delta}$ to take the average of the two structures $(J,\theta)$ and $(\mathring{J},\theta^{u})$ on $U$, instead of the standard structure $(\mathring{J},\Theta)$, so that we have the first equation of (\ref{basass}) which will make sure later that we have Proposition \ref{mainpro}.

\subsection{Some uniform bounds} In this subsection, we provide some uniform bounds of invariant functions, which will be used in the proof of Proposition \ref{mainpro}. Define
\begin{equation*}
F^{\delta }=\left( \frac{1}{1-|\phi ^{\delta }|^{2}}\right) ^{\frac{1}{2}}.
\end{equation*}%
Since $|\phi^{\delta }|\leq |\phi |$ by (\ref{3.14}), we have $1\leq |F^{\delta }|\leq |F|$, and hence $|F^{\delta }|$ \textbf{has an uniform bound}. By (\ref{3.14}), we have $(v^{\delta })^{2}-u^{2}=(1-\chi _{\delta })(1-u^{2})$,
which together with Lemma \ref{kl}
implies
\begin{equation}
u^{2}-|u^{2}-1|\leq (v^{\delta })^{2}\leq u^{2}+|u^{2}-1|.  \label{pbf3}
\end{equation}%
This implies that $v^{\delta }$ \textbf{has an uniform bound}. 
It follows from the definition of $\phi^\delta$ in (\ref{3.14}) that, for $a,b\in
\{1,\bar{1}\}$, the derivatives of $\phi^\delta$ are given by
\begin{equation}
\begin{split}
\phi _{a}^{\delta }& =-(\chi _{\delta })_{a}\phi +(1-\chi _{\delta })\phi
_{a}, \\
\phi _{ab}^{\delta }& =-(\chi _{\delta })_{ab}\phi -(\chi _{\delta
})_{a}\phi _{b}-(\chi _{\delta })_{b}\phi _{a}+(1-\chi _{\delta })\phi _{ab}.
\end{split}
\label{pbf4}
\end{equation}%
By (\ref{basass}), (\ref{pbf4}) and Lemma {\ref{kl}, $|\phi _{ab}|$ \textbf{has an uniform upper bound} for each $a,b\in \{1,%
\bar{1}\}$.  Also, for $a,b\in
\{1,\bar{1}\}$, it follows from (\ref{3.14}) that the derivatives of $v^\delta$ are given by
\begin{equation}
\begin{split}
(v^{\delta })_{a}& =\frac{1}{2}\frac{(\chi _{\delta })_{a}(u^{2}-1)+\chi
_{\delta }(u^{2})_{a}}{v^{\delta }}, \\
(v^{\delta })_{ab}& =\frac{1}{2}\frac{(\chi _{\delta })_{ab}(u^{2}-1)+(\chi
_{\delta })_{a}(u^{2})_{b}+(\chi _{\delta })_{b}(u^{2})_{a}+\chi _{\delta
}(u^{2})_{ab}}{v^{\delta }} \\
&\hspace{4mm} -\frac{1}{4}\frac{\left( (\chi _{\delta })_{a}(u^{2}-1)+\chi _{\delta
}(u^{2})_{a}\right) \left( (\chi _{\delta })_{b}(u^{2}-1)+\chi _{\delta
}(u^{2})_{b}\right) }{(v^{\delta })^{3}}.
\end{split}
\label{pbf5}
\end{equation}%
For the same reason, (\ref{basass}), (\ref{pbf5}) together with
Lemma {\ref{kl} show that $|(v^{\delta })_{a}|,|(v^{\delta })_{ab}|$ \textbf{%
has an uniform upper bound} for each $a,b\in \{1,\bar{1}\}$. }}

\subsection{Statement and Proof of Proposition \ref{mainpro}.} The sequence we construct in (\ref{3.14}) satisfies suitable properties which we summarize in Proposition \ref{mainpro}.
\begin{prop}\label{mainpro}
The sequence $\{(\phi^{\delta},\theta^{\delta})\}$ converges to $%
(\phi,\Theta)$ in $C^{0}$. The corresponding Tanaka-Webster scalar curvature $%
R^{\phi^{\delta},\theta^{\delta}}$ also converges to $R^{\phi,\Theta}$ in $%
C^{0}$.
\end{prop}

\begin{proof}
From the construction of $(\phi ^{\delta },\theta ^{\delta })$ in (\ref{3.14}), and noting
that $\phi (0)=0$ and $u(0)=1$, one can show that $\{(\phi ^{\delta
},\theta ^{\delta })\}$ converges to $(\phi ,\Theta )$ in $C^{0}$. Therefore
we only need to show that $R^{\phi ^{\delta },\theta ^{\delta }}$ converges
to $R^{\phi ,\Theta }$ in $C^{0}$.

Since $\theta^{\delta}=(v^{%
\delta})^{2}\Theta$ where $(v^{\delta})^{2}=1+\chi_{\delta}(u^{2}-1)$,
it follows from  the transformation law  of the Tanaka-Webster scalar curvature (\ref{trlaw03})
that
\begin{equation}  \label{pbf6}
\begin{split}
|R^{\phi^{\delta},\theta^{\delta}}-R^{\phi,\Theta}|&=\left|4\frac{%
\Delta_{b}^{\phi^{\delta},\Theta}v^{\delta}}{(v^{\delta})^{3}}+\frac{%
R^{\phi^{\delta},\Theta}}{(v^{\delta})^{2}}-R^{\phi,\Theta}\right| \\
&\leq\left|4\frac{\Delta_{b}^{\phi^{\delta},\Theta}v^{\delta}}{%
(v^{\delta})^{3}}+\frac{R^{\phi^{\delta},\Theta}}{(v^{\delta})^{2}}%
-R^{\phi,\Theta}(0)\right|+\Big|R^{\phi,\Theta}(0)-R^{\phi,\Theta}\Big|.
\end{split}%
\end{equation}
From above, we know that $v^{\delta}$ has an uniform bound and all $|\phi_{ab}|,|(v^{%
\delta})_{a}|,|(v^{\delta})_{ab}|$ and $|F^{\delta}|$ has an uniform upper
bound. Using (\ref{pbf1}) with $\phi, u$ replaced by $%
\phi^{\delta}, v^{\delta}$ respectively, together with (\ref{basass}), (\ref%
{pbf3}) and Lemma \ref{kl}, we have
\begin{equation}  \label{kees}
\begin{split}
&\left|4\frac{\Delta_{b}^{\phi^{\delta},\Theta}v^{\delta}}{(v^{\delta})^{3}}+%
\frac{R^{\phi^{\delta},\Theta}}{(v^{\delta})^{2}}-R^{\phi,\Theta}(0)\right|
\\
\leq&\left|(F^{\delta})^{2}\left(\frac{4\mathring{\Delta}_{b}v^{\delta}}{%
(v^{\delta})^3}\right)-\frac{(F^{\delta})^{2}(\bar{\phi^{\delta}}_{\bar{1}%
\bar{1}}+\phi^{\delta}_{11})}{(v^{\delta})^{2}}-R^{\phi,\Theta}(0)\right|+C%
\delta, \\
\leq&C\left|\left(\frac{4\mathring{\Delta}_{b}v^{\delta}}{(v^{\delta})^3}%
\right)-\frac{(\bar{\phi^{\delta}}_{\bar{1}\bar{1}}+\phi^{\delta}_{11})}{%
(v^{\delta})^{2}}-R^{\phi,\Theta}(0)\right|+C\delta, \\
\leq&C\left|\chi_{\delta}\left(\frac{u}{v^{\delta}}\right)\left(\frac{4%
\mathring{\Delta}_{b}v^{\delta}}{u^3}\right)-\frac{(\bar{\phi^{\delta}}_{%
\bar{1}\bar{1}}+\phi^{\delta}_{11})}{(v^{\delta})^{2}}-R^{\phi,\Theta}(0)%
\right|+C\delta,\ \text{by}\ (\ref{pbf5})\ \text{and Lemma}\ \ref{kl}, \\
\leq&C\left|\chi_{\delta}\left(\frac{4\mathring{\Delta}_{b}u}{u^3}%
\right)-(1-\chi_{\delta})(\bar{\phi}_{\bar{1}\bar{1}}+\phi_{11})-R^{\phi,%
\Theta}(0)\right|+C\delta, \\
\leq&C\left|\chi_{\delta}\left(\frac{4\mathring{\Delta}_{b}u}{u^3}%
\right)+\chi_{\delta}(\bar{\phi}_{\bar{1}\bar{1}}+\phi_{11})\right|+\left|-(%
\bar{\phi}_{\bar{1}\bar{1}}+\phi_{11})-R^{\phi,\Theta}(0)\right|+C\delta, \\
\leq&C\delta,\ \ \text{by}\ (\ref{trlaw040}), (\ref{trlaw041})\ \text{and}\ (%
\ref{pbf1}).
\end{split}%
\end{equation}
for some positive constant $C$. Combining (\ref{pbf6}) and (\ref{kees}), we
complete the proof of the proposition.
\end{proof}

\subsection{Proof of Theorem A}

Now we are ready to prove Theorem A. It suffices to prove the following proposition.

\begin{prop}
\label{proflim} Let $\lambda(M,J^{\delta})$ be the CR Yamabe constant with
respect to $J^{\delta}$, we have
\begin{equation}
\lim_{\delta\rightarrow 0}\lambda(M,J^{\delta})=\lambda(M,J).
\end{equation}
\end{prop}
\begin{proof}
Recall that we have constructed a sequence $(J^{\delta},\theta^{\delta})$
which converges to $(J,\theta)$ in $C^0$. In addition, $(J^{\delta},\theta^{%
\delta})=(J,\theta)$ outside the ball $B(\delta)$, and
\begin{equation*}
\phi^{\delta}=(1-\chi_\delta(\rho))\phi,\ \theta^{\delta}=(v^{\delta})^{2}\Theta \
\ \text{in}\ B(\delta).
\end{equation*}
Notice that we have chosen the contact form $\theta$ such that $%
\theta|_{B(\delta)}=\Theta$. Let $dV=\theta\wedge d\theta$ and $%
dV^{\delta}=\theta^{\delta}\wedge d\theta^{\delta}=(v^{\delta})^{4}dV$. And
let $R=R^{J,\theta}=R^{\phi,\Theta}$ and $R^{\delta}=R^{J^{\delta},\theta^{%
\delta}}=R^{\phi^{\delta},\theta^{\delta}}$. Since $(J^{\delta},\theta^{%
\delta})\rightarrow(J,\theta)$ in $C^0$, it is easy to see that, for any $%
\varepsilon>0$ with $\varepsilon\ll 1$, if $\delta$ is small enough, then we have
\begin{equation}\label{3.22}
|(v^{\delta})^{\pm4}-1|\leq \varepsilon,\ \ |R-R^{\delta}|\leq\varepsilon,
\end{equation}
and hence $|R^{\delta}|$ \textbf{has an uniform bound}. Next we need the
following:

\begin{lem}
\label{kele1} Given $\varepsilon>0$, if $\delta$ is small enough, then we
have
\begin{equation}
\frac{1}{(1+\varepsilon)}|\nabla^{\delta}_{b}u|_{\delta}^{2}\leq|%
\nabla_{b}u|^{2}\leq(1+\varepsilon)|\nabla^{\delta}_{b}u|_{\delta}^{2}.
\end{equation}
\end{lem}
\begin{proof}[Proof of Lemma \protect\ref{kele1}:]
First, by (\ref{defof}), we have
\begin{equation*}
\begin{split}
|\nabla_{b}u|^{2}&=2F^{2}((1+|\phi|^{2})|u_{\bar 1}|^{2}+\phi(u_{1})^{2}+%
\bar{\phi}(u_{\bar 1})^{2}), \\
|\nabla^{\delta}_{b}u|_{\delta}^{2}&=2(F^{\delta})^{2}((1+|\phi^{%
\delta}|^{2})|u_{\bar 1}|^{2}+\phi^{\delta}(u_{1})^{2}+\bar{\phi}%
^{\delta}(u_{\bar 1})^{2})(v^{\delta})^{-2}.
\end{split}%
\end{equation*}
Therefore, whenever $u_{\bar 1}\neq 0$, we have
\begin{equation*}
\begin{split}
\frac{|\nabla_{b}u|^{2}}{|\nabla^{\delta}_{b}u|_{\delta}^{2}}%
&=(v^{\delta})^{2}\frac{F^{2}\left((1+|\phi|^{2})+\phi\frac{(u_{1})^{2}}{%
|u_{\bar 1}|^{2}}+\bar{\phi}\frac{(u_{\bar 1})^{2}}{|u_{\bar 1}|^{2}}\right)%
}{(F^{\delta})^{2}\left((1+|\phi^{\delta}|^{2})+\phi^{\delta}\frac{%
(u_{1})^{2}}{|u_{\bar 1}|^{2}}+\bar{\phi}^{\delta}\frac{(u_{\bar 1})^{2}}{%
|u_{\bar 1}|^{2}}\right)} \\
&\leq(v^{\delta})^{2}\frac{F^{2}(1+|\phi|^{2}+2|\phi|)}{(F^{\delta})^{2}(1+|%
\phi^{\delta}|^{2}-2|\phi^{\delta}|)} \\
&=(v^{\delta})^{2}\frac{F^{2}(1+|\phi|)^{2}}{(F^{\delta})^{2}(1-|\phi^{%
\delta}|)^{2}} \\
&\leq 1+%
\varepsilon,
\end{split}%
\end{equation*}
if $\delta$ is small enough (since $\phi(0)=0$). Similarly, we have
\begin{equation*}
\frac{|\nabla^{\delta}_{b}u|_{\delta}^{2}}{|\nabla_{b}u|^{2}}%
\leq 1+\varepsilon.
\end{equation*}
We have thus completed the proof of Lemma \ref{kele1}.
\end{proof}

Now, for each $\delta>0$, there exists a function $u^{\delta}$ such that
\begin{equation*}
\int_{M}(u^{\delta})^4 dV^{\delta}=1~~\mbox{ and }~~
\lambda(M,J^{\delta})\leq
E_{\theta^{\delta}}(u^{\delta})\leq\lambda(M,J^{\delta})+\varepsilon.
\end{equation*}
We are going to estimate
\begin{equation}  \label{qoe1}
\begin{split}
E_{\theta}(u^{\delta})&=\int_{M}(4|\nabla_{b}u^{\delta}|^{2}+R(u^{%
\delta})^{2})\ dV \\
&=E_{\theta^{\delta}}(u^{\delta})+4\left(\int_{M}|\nabla_{b}u^{\delta}|^{2}\
dV-\int_{M}|\nabla^{\delta}_{b}u^{\delta}|_{\delta}^{2}\ dV^{\delta}\right)
\\
&\ \ +\left(\int_{M}R(u^{\delta})^{2}dV-\int_{M}R^{\delta}(u^{\delta})^{2}\
dV^{\delta}\right).
\end{split}%
\end{equation}
By (\ref{3.22}) and H\"{o}lder inequality, we have
\begin{equation}  \label{qoe2}
\begin{split}
&\int_{M}R(u^{\delta})^{2}dV-\int_{M}R^{\delta}(u^{\delta})^{2}\ dV^{\delta}
\\
=&\int_{M}R(u^{\delta})^{2}(dV-dV^{\delta})+\int_{M}(R-R^{\delta})(u^{%
\delta})^{2}\ dV^{\delta} \\
=&\int_{M}R(u^{\delta})^{2}((v^{\delta})^{-4}-1)dV^{\delta}+\int_{M}(R-R^{%
\delta})(u^{\delta})^{2}\ dV^{\delta} \\
\leq&C\varepsilon,\ \ \text{for some positive constant}\ C,
\end{split}%
\end{equation}
and, by (\ref{3.22}) and Lemma \ref{kele1}, we have
\begin{equation}  \label{qoe3}
\begin{split}
&\int_{M}|\nabla_{b}u^{\delta}|^{2}\
dV-\int_{M}|\nabla^{\delta}_{b}u^{\delta}|_{\delta}^{2}\ dV^{\delta} \\
=&\int_{M}\left(|\nabla_{b}u^{\delta}|^{2}(v^{\delta})^{-4}-|\nabla^{%
\delta}_{b}u^{\delta}|_{\delta}^{2}\right)\ dV^{\delta} \\
\leq&\int_{M}\left((1+\varepsilon)|\nabla^{\delta}_{b}u^{\delta}|_{%
\delta}^{2}(v^{\delta})^{-4}-|\nabla^{\delta}_{b}u^{\delta}|_{\delta}^{2}%
\right)\ dV^{\delta} \\
=&\int_{M}|\nabla^{\delta}_{b}u^{\delta}|_{\delta}^{2}\Big(\lbrack
(v^{\delta})^{-4}-1]+\varepsilon(v^{\delta})^{-4}\Big)\ dV^{\delta} \\
\leq&C\varepsilon E_{\theta^{\delta}}(u^{\delta}),\ \ \text{for some}\ C>0,
\end{split}%
\end{equation}
since $|\int_{M}R^{\delta}(u^{\delta})^{2}\ dV^{\delta}|$ has an uniform
upper bound. Substituting (\ref{qoe2}) and (\ref{qoe3}) into (\ref{qoe1}),
we obtain
\begin{equation}  \label{qoe4}
E_{\theta}(u^{\delta})\leq(1+C\varepsilon)E_{\theta^{\delta}}(u^{\delta})+C%
\varepsilon \leq(1+C\varepsilon)(\lambda(M,J^{\delta})+\varepsilon)+C\varepsilon.
\end{equation}
Similarly, we have
\begin{equation}  \label{qoe5}
E_{\theta^{\delta}}(u^{0})\leq(1+C\varepsilon)E_{\theta}(u^{0})+C%
\varepsilon
\leq(1+C\varepsilon)(\lambda(M,J)+\varepsilon)+C\varepsilon,
\end{equation}
where $u^{0}$ is a function such that
\begin{equation*}
\int_{M}(u^{0})^4 dV=1~~\mbox{ and }~~
\lambda(M,J)\leq E_{\theta}(u^{0})\leq\lambda(M,J)+\varepsilon.
\end{equation*}
Since $\int_{M}(u^{0})^{4}dV=1$, we have
\begin{equation}  \label{qoe6}
1-\varepsilon\leq\int_{M}(u^{0})^{4}dV^{\delta}\leq1+\varepsilon
\end{equation}
by (\ref{3.22}).
Similarly, we have
\begin{equation}  \label{qoe7}
1-\varepsilon\leq\int_{M}(u^{\delta})^{4}dV\leq1+\varepsilon.
\end{equation}
By  (\ref{qoe4})-(\ref{qoe7}), we can get
\begin{equation}
\frac{(1-\varepsilon)\lambda(M,J)-\varepsilon(1+C\varepsilon)+C\varepsilon}{%
1+C\varepsilon}\leq\lambda(M,J^{\delta})\leq\frac{(1+C\varepsilon)(%
\lambda(M,J)+\varepsilon)+C\varepsilon}{1-\varepsilon}.
\end{equation}
This completes the proof of Proposition \ref{proflim}.
\end{proof}

Therefore, if $\delta$ is small enough,
we have $\lambda (M,J^{\delta })>0$
by Proposition \ref{proflim}
and the assumption that $\lambda (M,J)>0$. By construction, each $J^{\delta }$ is spherical around the point $p\in M$. To
complete the proof of {\bf Theorem A}, we choose $\delta _{1},\delta _{2}$ so that
the CR Yamabe constants $\lambda (M_{1},J_{1}^{\delta _{1}})$ and $\lambda
(M_{2},J_{2}^{\delta _{2}})$ are both positive. Then, by using the argument in \cite%
{Cheng&Chiu} (see the paragraph after Theorem \ref{thm1}) to glue $M_{1}$
and $M_{2}$ by a Heisenberg cylinder, we get a CR structure on the connected
sum $M_{1}\#M_{2}$ with positive  CR Yamabe constant.

\bibliographystyle{plain}

\end{document}